\documentclass[a4paper,11pt]{article}
\usepackage[top=3cm,bottom=3cm,left=2.5cm,right=2.5cm]{geometry}
\usepackage[algoruled,linesnumbered]{algorithm2e}
\usepackage{algorithmic}

\usepackage{amsfonts,epsf,amsmath,amssymb,here, graphicx}
\usepackage{latexsym,multirow,rotating}
\usepackage{pgf,tikz,subfigure}
\usepackage{epstopdf}

\newtheorem{theorem}{\bf Theorem}[section]
\newtheorem{corollary}[theorem]{\bf Corollary}
\newtheorem{lemma}[theorem]{\bf Lemma}

\newtheorem{remark}[theorem]{\bf Remark}
\newtheorem{definition}[theorem]{\bf Definition}

\newcommand{\qed}{\hfill $\square$ \bigskip}

\begin{document}

\baselineskip=0.30in
\vspace*{40mm}

\begin{center}
{\LARGE \bf\sc The Edge-Szeged Index and the PI Index of Benzenoid Systems in Linear Time}
\bigskip \bigskip

{\large \bf Niko Tratnik
}
\bigskip\bigskip

\baselineskip=0.20in

Faculty of Natural Sciences and Mathematics, University of Maribor, Slovenia \\
{\tt niko.tratnik1@um.si}
\medskip

\bigskip\medskip

(Received August 2, 2016)

\end{center}

\noindent
{\bf Abstract}

\vspace{3mm}\noindent
The edge-Szeged index of a graph $G$ is defined as $Sz_e(G) = \sum_{e=uv \in E(G)}m_u(e)m_v(e)$, where $m_u(e)$ denotes the number of edges of $G$ whose distance to $u$ is smaller than the distance to $v$ and $m_v(e)$ denotes the number of edges of $G$ whose distance to $v$ is smaller than the distance to $u$. Similarly, the PI index is defined as $PI(G) = \sum_{e=uv \in E(G)}(m_u(e) + m_v(e))$. In this paper it is shown how the problem of calculating the indices of a benzenoid system can be reduced to the problem of calculating weighted indices of three different weighted quotient trees. Furthermore, using these results, algorithms are established that, for a given benzenoid system $G$ with $m$ edges, compute the edge-Szeged index and the PI index of $G$ in $O(m)$ time. Moreover, it is shown that the results can also be applied to weighted benzenoid systems.

\vspace{5mm}

\baselineskip=0.30in



\section{Introduction}
In the present paper we study some distance-based topological indices, which are closely related to the Wiener index. Their history goes back to $1947$, when H. Wiener used the distances in the molecular graphs of alkanes to calculate their boiling points \cite{wiener}. This research has led to the Wiener index, which is defined as
$$W(G) = \sum_{\lbrace u, v \rbrace \subseteq V(G)} d_G(u,v)$$
and is one of the most popular molecular descriptors. The Wiener index, due to its correlation with a large number of physico-chemical properties of organic molecules and its interesting mathematical properties, has been extensively studied in both theoretical and chemical literature. Nowadays this index is used for preliminary screening of drug molecules and for predicting binding energy of protein-ligand complex at a preliminary stage. Furthermore, the same quantity has been studied by mathematicians as the gross status, the distance of graphs and the transmission in graphs. 

Later, the Szeged index and the PI index were introduced and it was shown that they also have many applications, for example in drug modelling \cite{drug}, in networks \cite{klavzar-2013,pisanski} and in biological activities of a large number of diversified and complex compounds \cite{pi}. 

It is known that if $T$ is a tree, then
$$W(T) = \sum_{e =uv \in E(T)}n_u(e)n_v(e),$$
where $n_u(e)$ denotes the number of vertices of $T$ whose distance to $u$ is smaller than the distance to $v$ and $n_v(e)$ denotes the number of vertices of $T$ whose distance to $v$ is smaller than the distance to $u$. Therefore, a proper generalization was defined in \cite{gut_sz} as

$$Sz(G) = \sum_{e=uv \in E(G)}n_u(e)n_v(e).$$
This structural descriptor was eventually named as the Szeged index. Motivated by the success of the Szeged index, in \cite{def_pi} a seemingly similar molecular descriptor, that is called the PI index (or the edge-PI index) was defined with
$$PI_e(G) = \sum_{e=uv \in E(G)}\big(m_u(e) + m_v(e)\big),$$
where the numbers $m_u(e)$ and $m_v(e)$ are the edge-variants
of the numbers $n_u(e)$ and $n_v(e)$. Later \cite{khal1}, a vertex version of the PI index, called the vertex-PI index was defined as
$$PI_v(G) = \sum_{e=uv \in E(G)}\big(n_u(e) + n_v(e)\big).$$
Obviously, for any bipartite graph it follows $PI_v(G) = |V(G)||E(G)|$ (see \cite{khal1} and Lemma \ref{alg31}). Since benzenoid systems are bipartite graphs, the vertex-PI index can be easily computed by multiplying the number of vertices and the number of edges.

Finally, the edge-version of the Szeged index, the edge-Szeged index, was defined in \cite{gut} as
$$Sz_e(G) = \sum_{e=uv \in E(G)}m_u(e)m_v(e).$$
It is also known that if $T$ is a tree, then the edge-Szeged index $Sz_e(T)$ equals the edge-Wiener index $W_e(T)$ (for the details see \cite{azari}). Therefore, the edge-Szeged index can be viewed as a generalization of the edge-Wiener index. The Wiener index, the Szeged index, the edge-Szeged index and the PI index are some of the central and most commonly studied distance-based topological indices. For some recent research on the edge-Szeged index see \cite{al-fozan,ashrafi,khal,khalifeh-2009,azari} and on the PI indices see \cite{ilic,john,khalifeh-2009,klavzar-2007,koor}.

In the present paper linear time algorithms for the edge-Szeged index and the PI index of benzenoid systems are developed.  Benzenoid systems form one of the most extensively studied family of chemical graphs. Although similar results could also be obtained for the vertex-PI index, they are not needed since the computation of this index in the case of bipartite graphs is trivial. However, in \cite{klavzar-2007} the formula to calculate the PI index using PI partitions (orthogonal cuts in the case of benzenoid systems) was obtained and in \cite{john} it was shown how the cut method can be used to calculate the PI index. Therefore, following these results, a fast algorithm that calculates the PI index of a benzenoid system could be obtained. But the purpose of this paper is also to show the connection between the PI index of a benzenoid system $G$ and the weighted indices of corresponding weighted trees. Also, a method for calculating the edge-Szeged index of benzenoid chains was obtained in \cite{wang}.

The algorithms are parallel to the linear time algorithms for the Wiener index, the Szeged index and the edge-Wiener index developed in \cite{chepoi-1997,kelenc}. We use the cut method to reduce the problem of calculating the indices of benzenoid systems to the problem of calculating weighted indices of three different weighted quotient trees (for more information about the cut method see \cite{klavzar-2008,klavzar-2015}). Some of the corresponding weighted indices of the edge-Szeged index, the PI index and the vertex-PI index have been defined in different ways (for example, see \cite{ilic, khal}). In this paper we carefully define the corresponding weighted indices such that they are suitable for our purpose.

The paper is organized in $6$ sections. In the next section we give definitions and introduce some basic notation needed later. In Section \ref{sec:theory} we use the cut method to prove that the edge-Szeged index and the PI index of a benzenoid system can be expressed as the sum of corresponding weighted indices of weighted trees. Using these results the algorithms are presented in Section~\ref{sec:algorithm}. In Section \ref{example} it is shown how the results from Section \ref{sec:theory} can be used to calculate the indices by hand. In the last section a generalization of the method to weighted graphs is described.

\section{Preliminaries}

Unless stated otherwise, the graphs considered in this paper are simple, finite and connected. We define $d_G(u,v)$ to be the usual shortest-path distance between vertices $u, v\in V(G)$. In addition, for a vertex $x \in V(G)$ and an edge $e=ab \in E(G)$ we set 
\begin{equation*}
d_G(x,e) = \min \lbrace d_G(x,a), d_G(x, b) \rbrace\,.
\end{equation*}
Let $G$ be a graph and $e=uv$ an edge of $G$. Throughout the paper we will use the following notation:
$$N_1(e|G) = \lbrace x \in V(G) \ | \ d_G(x,u) < d_G(x,v) \rbrace, $$
$$N_2(e|G) = \lbrace x \in V(G) \ | \ d_G(x,v) < d_G(x,u) \rbrace, $$
$$M_1(e|G) = \lbrace f \in E(G) \ | \ d_G(f,u) < d_G(f,v) \rbrace, $$
$$M_2(e|G) = \lbrace f \in E(G) \ | \ d_G(f,v) < d_G(f,u) \rbrace. $$
The {\em Szeged index} of a graph $G$ is defined as $Sz(G) =\sum_{e \in E(G)}|N_1(e|G)| \cdot |N_2(e|G)|$. To emphasize that it is the vertex-Szeged index, we will also write $Sz_v(G)$ for $Sz(G)$. The edge-Szeged index is defined with the formula
$$Sz_e(G) =\sum_{e \in E(G)}|M_1(e|G)| \cdot |M_2(e|G)|.$$

\noindent
The \textit{PI index} and the \textit{vertex-PI index} of a graph $G$ are defined as
\begin{eqnarray*}
PI_e(G) & = & \sum_{e \in E(G)}\big(|M_1(e|G)| + |M_2(e|G)|\big), \\
PI_v(G) & = & \sum_{e \in E(G)}\big(|N_1(e|G)| + |N_2(e|G)|\big).
\end{eqnarray*}

We extend the above definitions of the indices to weighted graphs as follows. Let $G$ be a graph and let $w:V(G)\rightarrow {\mathbb R}^+$  and $w':E(G)\rightarrow {\mathbb R}^+$ be given functions. Then $(G,w')$ and $(G,w,w')$ are an {\em edge-weighted graph} and a {\em vertex-edge weighted graph}, respectively. Furthermore, if $e=uv$ is an edge of $G$ and $i \in \lbrace 1, 2\rbrace$, set
\begin{eqnarray*}
n_i(e|G) & = & \sum_{x \in N_i(e|G)} w(x), \\
m_i(e|G) & = & \sum_{f \in M_i(e|G)} w'(f), \\
r_i(e|G) & = & \sum_{x \in N_i(e|G)} w(x) + \sum_{f \in M_i(e|G)} w'(f). 
\end{eqnarray*}

\noindent
Then the Szeged index, the edge-Szeged index, and the vertex-edge-Szeged index of these weighted graphs are defined as
\begin{eqnarray*}
Sz_v(G,w,w') & = & \sum_{e \in E(G)} w'(e)n_1(e|G)n_2(e|G), \\
Sz_e(G,w') & = & \sum_{e \in E(G)} w'(e)m_1(e|G)m_2(e|G), \\
Sz_{ve}(G,w,w') & = & \sum_{e \in E(G)} w'(e)\big(n_1(e|G)m_2(e|G) + n_2(e|G)m_1(e|G)\big). 
\end{eqnarray*}

\noindent
The corresponding PI indices of these weighted graphs are defined as
\begin{eqnarray*}
PI_e(G,w') & = & \sum_{e \in E(G)} w'(e)\big(m_1(e|G) + m_2(e|G)\big)\\
PI_v(G,w,w') & = & \sum_{e \in E(G)} w'(e)\big(n_1(e|G) + n_2(e|G)\big).
\end{eqnarray*}

\begin{remark}
Let $(T,w,w')$ be a weighted tree. For an edge $e$ of $T$, let $T_1$ and $T_2$ be the connected components of $T \setminus e$. Obviously, for $i=1,2$ it holds
\begin{eqnarray*}
n_i(e|T) & = & \sum_{u \in V(T_i)}w(u), \\
m_i(e|T) & = & \sum_{f \in E(T_i)}w'(f), \\
r_i(e|T) & = & \sum_{u \in V(T_i)}w(u) + \sum_{f \in E(T_i)}w'(f).
\end{eqnarray*}

\end{remark}

Let ${\cal H}$ be the hexagonal (graphite) lattice and let $Z$ be a cycle on it. Then a {\em benzenoid system} is induced by the vertices and edges of ${\cal H}$, lying on $Z$ and in its interior. The edge set of a benzenoid system $G$ can be naturally partitioned into sets $E_1, E_2$, and $E_3$ of edges of the same direction. For $i \in \lbrace 1, 2, 3 \rbrace$, set $G_i = G - E_i$. Then the connected components of the graph $G_i$ are paths. The quotient graph $T_i$, $1\le i\le 3$, has these paths as vertices, two such paths (i.e. components of $G_i$) $P_1$ and $P_2$ being adjacent in $T_i$ if some edge in $E_i$ joins a vertex of $P_1$ to a vertex of $P_2$. It is known that $T_1$, $T_2$ and $T_3$ are trees (see \cite{chepoi-1996}).

\noindent
An \textit{elementary cut} of a benzenoid system $G$ is a line segment that starts at
the center of a peripheral edge of a benzenoid system $G$,
goes orthogonal to it and ends at the first next peripheral
edge of $G$. Details on elementary cuts can be found elsewhere (for example, see \cite{gut}). In what follows we denote by $C$ the set of edges in an elementary cut. The number of edges in $C$ will be denoted in a standard way, by $|C|$. 

\noindent
The set of all elementary cuts of $G$ will be denoted by $\mathcal{C}$. Furthermore, if $i \in \lbrace 1, 2, 3 \rbrace$, the set $\mathcal{C}_i$ denotes the set of all elementary cuts $C$ of $G$ such that $C \subseteq E_i$. Then $\lbrace \mathcal{C}_1, \mathcal{C}_2, \mathcal{C}_3 \rbrace$ is a partition of the set $\mathcal{C}$. It is obvious that there is a natural bijection between elementary cuts in $\mathcal{C}_i$ and edges of $T_i$.

\section{Reduction to the quotient trees}
\label{sec:theory}
Let $G$ be a benzenoid system and let $T_1$, $T_2$, $T_3$ be its quotient trees as defined in the preliminaries. We next extend the quotient trees $T_i$, $i \in \lbrace 1, 2, 3 \rbrace$, to weighted trees $(T_i, w'_i)$ and $(T_i, w_i, w'_i)$ as follows: 
\begin{itemize}
\item for $x \in V(T_i)$, let $w_i(x)$ be the number of edges in the component (path) $x$ of $G_i$;
\item for $e = xy \in E(T_i)$, let $w_i'(e)$ be the number of edges between components (paths) $x$ and $y$.
\end{itemize}

\noindent
Obviously, if $C$ is an elementary cut corresponding to the edge $e \in E(T_i)$, then $|C| = w_i'(e)$.

\noindent
We start with the following lemma, which is used in the proof of Theorem \ref{sz_sum} and Theorem \ref{pi_sum}. If $C$ is an elementary cut of $G$, then it is clear that for every two edges $f_1$ and $f_2$ in $C$ it follows (without loss of generality) $M_1(f_1|G) = M_1(f_2|G)$ and $M_2(f_1|G) = M_2(f_2|G)$. Therefore, we can define $M_1(C|G) = M_1(f|G)$ and $M_2(C|G) = M_2(f|G)$ for some $f \in C$.

\begin{lemma}
\label{lema}
Let $G$ be a benzenoid system and let $C$ be an elementary cut of $G$ connecting paths $P_1$ and $P_2$ of $G-E_i$ for $i \in \lbrace 1, 2, 3\rbrace$. If $P_1$ corresponds to vertex $u \in V(T_i)$ and $P_2$ corresponds to vertex $v \in V(T_i)$ and $e=uv \in E(T_i)$, then it holds (without loss of generality)
$$|M_1(C|G)| = n_1(e|T_i) + m_1(e|T_i)$$
and
$$|M_2(C|G)| = n_2(e|T_i) + m_2(e|T_i).$$
\end{lemma}

\begin{proof}
Let $C$ be an elementary cut of $G$ connecting paths $P_1$ and $P_2$ of $G-E_i$ for $i \in \lbrace 1, 2, 3\rbrace$. 
Obviously, the number $n_1(e|T_i)$ represents the number of edges of $G$ corresponding to the vertices of $T_i$ that are closer to $u$ than to $v$. Analogously, the number $m_1(e|T_i)$ represents the number of edges of $G$ corresponding to the edges of $T_i$ that are closer to $u$ than to $v$. Furthermore, the edges and vertices of $T_i$ that are closer to $u$ than to $v$ represent the paths and elementary cuts of $G-E_i$, the edges of which are in the set $M_1(C|G)$. Hence, $|M_1(C|G)| = n_1(e|T_i) + m_1(e|T_i)$. To prove that $|M_2(C|G)| = n_2(e|T_i) + m_2(e|T_i)$ we can use very similar reasoning. Therefore, we are done. \qed
\end{proof}

\subsection{The edge-Szeged index}

The following theorem is the basis for the computation of the edge-Szeged index.
\begin{theorem}
\label{sz_sum}
If $G$ is a benzenoid system and $(T_i,w_i,w_i')$, $1\le i\le 3$, are the corresponding weighted quotient trees, then
$$Sz_e(G)= \sum_{i=1}^{3}\left( Sz_v(T_i,w_i,w_i') + Sz_e(T_i,w_i') + Sz_{ve}(T_i,w_i,w_i')\right).$$
\end{theorem}

\begin{proof}
We have
$$Sz_e(G) = \sum_{e \in E(G)}|M_1(e|G)|\cdot |M_2(e|G)| = \sum_{C \in \mathcal{C}}|C|\cdot |M_1(C|G)|\cdot |M_2(C|G)|,$$

\noindent
since for every two edges $f_1,f_2 \in C$ it holds $M_1(f_1|G)=M_1(f_2|G)$ and $M_2(f_1|G)=M_2(f_2|G)$ (see also \cite{gut}). Recall that the set $\lbrace \mathcal{C}_1, \mathcal{C}_2, \mathcal{C}_3 \rbrace $ is a partition of $\mathcal{C}$ and there exists a bijection between elements of $\mathcal{C}_i$ and edges of $T_i$ such that $|C|=w_i'(e)$ if an elementary cut $C$ corresponds to the edge $e$ of $T_i$. Using Lemma \ref{lema} we obtain

\begin{eqnarray*}
Sz_e(G) & = & \sum_{i=1}^{3}\sum_{C \in \mathcal{C}_i}|C|\cdot |M_1(C|G)|\cdot |M_2(C|G)| = \\
  & = & \sum_{i=1}^{3}\sum_{e \in E(T_i)}w_i'(e)\big(n_1(e|T_i) + m_1(e|T_i)\big)\big(n_2(e|T_i) + m_2(e|T_i)\big)= \\
  & = & \sum_{i=1}^{3}\sum_{e \in E(T_i)}w_i'(e)n_1(e|T_i)n_2(e|T_i) + \sum_{i=1}^{3}\sum_{e \in E(T_i)}w_i'(e)m_1(e|T_i)m_2(e|T_i) + \\
  & + & \sum_{i=1}^{3}\sum_{e \in E(T_i)}w_i'(e)\left(n_1(e|T_i)m_2(e|T_i) + n_2(e|T_i)m_1(e|T_i) \right) = \\
  & = & \sum_{i=1}^{3}\left( Sz_v(T_i,w_i,w_i') + Sz_e(T_i,w_i') + Sz_{ve}(T_i,w_i,w_i')\right).
\end{eqnarray*} \qed
\end{proof}

\noindent
In order to simplify the notation and the algorithm for computing the edge-Szeged index of a benzenoid system, we introduce a new concept -  the weighted total-Szeged index.

\begin{definition}
Let $(G,w,w')$ be a weighted graph. The weighted total-Szeged index of a graph $G$ is defined as
$$Sz_t(G,w,w') = \sum_{e \in E(G)}w'(e)r_1(e|G)r_2(e|G).$$
\end{definition}

\noindent
The following lemma follows easily.
\begin{lemma}
\label{szeged}
If $(G,w,w')$ is a weighted graph, then 
$$Sz_t(G,w,w')=Sz_v(G,w,w') + Sz_e(G,w') + Sz_{ve}(G,w,w').$$
\end{lemma}
\begin{proof}
For any edge $e \in E(G)$ it holds
$$r_1(e|G)r_2(e|G) = (n_1(e|G) + m_1(e|G))(n_2(e|G) + m_2(e|G)) =$$ $$= n_1(e|G)n_2(e|G) + m_1(e|G)m_2(e|G) + \big(n_1(e|G)m_2(e|G) + n_2(e|G)m_1(e|G)\big)$$
and the proof is complete. \qed
\end{proof}

\noindent
Using Theorem \ref{sz_sum} and Lemma \ref{szeged}, we obtain the following corollary.
\begin{corollary}
\label{sz_cor}
If $G$ is a benzenoid system and $(T_i,w_i,w_i')$, $1\le i\le 3$, are the corresponding weighted quotient trees, then
$$ Sz_e(G)= Sz_t(T_1,w_1,w_1') + Sz_t(T_2, w_2,w_2') + Sz_t(T_3, w_3, w_3').$$
\end{corollary}

\subsection{The PI index}

The following theorem shows how the problem of computing the PI index of a benzenoid system can be reduced to the problem of calculating the weighted indices of the corresponding weighted quotient trees.
\begin{theorem}
\label{pi_sum}
If $G$ is a benzenoid system and $(T_i,w_i,w_i')$, $1\le i\le 3$, are the corresponding weighted quotient trees, then
$$PI_e(G)= \sum_{i=1}^{3}\big( PI_e(T_i,w_i') + PI_v(T_i,w_i,w_i') \big).$$
\end{theorem}

\begin{proof}
We have
$$PI_e(G) = \sum_{e \in E(G)}\big(|M_1(e|G)| + |M_2(e|G)|\big) = \sum_{C \in \mathcal{C}}|C|\big(|M_1(C|G)| + |M_2(C|G)| \big),$$

\noindent
since for every two edges $f_1,f_2 \in C$ it holds $M_1(f_1|G)=M_1(f_2|G)$ and $M_2(f_1|G)=M_2(f_2|G)$. Using very similar arguments as in the proof of Theorem \ref{sz_sum}, we obtain

\begin{eqnarray*}
PI_e(G) & = & \sum_{i=1}^{3}\sum_{C \in \mathcal{C}_i}|C|\big( |M_1(C|G)| + |M_2(C|G)|\big) = \\
  & = & \sum_{i=1}^{3}\sum_{e \in E(T_i)}w_i'(e)\big(n_1(e|T_i) + m_1(e|T_i) + n_2(e|T_i) + m_2(e|T_i)\big)= \\
  & = & \sum_{i=1}^{3}\sum_{e \in E(T_i)}w_i'(e)\big(n_1(e|T_i) + n_2(e|T_i)\big) + \sum_{i=1}^{3}\sum_{e \in E(T_i)}w_i'(e)\big(m_1(e|T_i) + m_2(e|T_i)\big) = \\
  & = & \sum_{i=1}^{3}\big( PI_v(T_i,w_i,w_i') + PI_e(T_i,w_i') \big).
\end{eqnarray*}
\qed
\end{proof}

\section{The algorithms}
\label{sec:algorithm}

In this section we describe the announced algorithms. For a given benzenoid system $G$ we first compute the quotient trees $T_i$ using a procedure called {\tt QuotientTrees}. For each $T_i$ we then compute the vertex weights $w$ and the edge weights $w'$ using a procedure {\tt Weights}. It was proved in \cite{chepoi-1996} that this can be done in linear time. The algorithms thus read as follows:
\vspace{1cm}

\begin{algorithm}[H]\label{alg:szeged}
\SetKwInOut{Input}{Input}\SetKwInOut{Output}{Output}
\DontPrintSemicolon

\Input{Benzenoid system $G$ with $m$ edges}
\Output{$Sz_e(G)$}

 \SetKwFunction{CQT}{QuotientTrees}
 \SetKwFunction{CW}{Weights}
 $(T_1,T_2,T_3) \leftarrow$ \CQT($G$)\;
 \For{$i=1$  {\rm to} $3$}{
 	$(w_i,w'_i) \leftarrow $ \CW($T_i,G$)\;
}
 	$X_1 \leftarrow Sz_t(T_1,w_1,w'_1)$\;
 	$X_2 \leftarrow Sz_t(T_2,w_2,w'_2)$\;
 	$X_3 \leftarrow Sz_t(T_3,w_3,w'_3)$\;
 $Sz_e(G) \leftarrow X_1+X_2+X_3$
 \caption{Edge-Szeged Index of Benzenoid Systems}
\end{algorithm} 

\vspace{0.5cm}

\begin{algorithm}[H]\label{alg:pi}
\SetKwInOut{Input}{Input}\SetKwInOut{Output}{Output}
\DontPrintSemicolon
 
\Input{Benzenoid system $G$ with $m$ edges}
\Output{$PI_e(G)$}

 \SetKwFunction{CQT}{QuotientTrees}
 \SetKwFunction{CW}{Weights}
 $(T_1,T_2,T_3) \leftarrow$ \CQT($G$)\;
 \For{$i=1$  {\rm to} $3$}{
 	$(w_i,w'_i) \leftarrow $ \CW($T_i,G$)\;
 	$X_{i,1} \leftarrow PI_v(T_i,w_i,w'_i)$\;
 	$X_{i,2} \leftarrow PI_e(T_i,w'_i)$\;
 	$Y_i \leftarrow X_{i,1}+X_{i,2}$
 }
 $PI_e(G) \leftarrow Y_1+Y_2+Y_3$
 \caption{PI Index of Benzenoid Systems}
\end{algorithm} 

%
\vspace{1cm}

\noindent
The correctness of the algorithms obviously follows from Corollary \ref{sz_cor} and Theorem \ref{pi_sum}. Therefore, we have to check the time complexity. Recall that weighted quotient trees $(T_i,w_i,w_i')$ can be computed in linear time. The following three lemmas show that the corresponding weighted indices of weighted trees can be computed in linear time. 

\begin{lemma}
\label{alg1}
Let $(T,w,w')$ be a weighted tree with $m$ edges. Then the weighted total-Szeged index $Sz_t(T,w,w')$ can be computed in $O(m)$ time.
\end{lemma}

\begin{proof}
We will use a method parallel to the method from \cite{chepoi-1997} and \cite{kelenc}. Let $T$ be a rooted tree with a root $x$ and label the vertices of $T$ such that, if a vertex $y$ is labelled $\ell$, then all vertices in the subtree rooted at $y$ have labels smaller than $\ell$. Using the standard BFS (breadth-first search) algorithm this can be done in linear time. Next we visit all the vertices of $T$ according to this labelling (we start with the smallest label). Furthermore, for every vertex $y$, we will adopt weights $w(y)$ and calculate new weights $w''(y)$ and $s(y)$.

\noindent
Assume that we are visiting vertex $y \in V(T)$. The new weight $w(y)$ will be computed as the sum of all the weights of vertices in the subtree rooted at $y$ and $w''(y)$ will be computed as the sum of all weights of the edges in the subtree rooted at $y$. 
\begin{itemize}
\item If $y$ is a leaf, then $w(y)$ is left unchanged and $w''(y)=0$.
\item If $y$ is not a leaf, then update $w(y)$ by adding to it $w(z)$ for all down-neighbours $z$ of $y$, and compute $w''(y)$ as the sum of $w''(z)$ and $w'(e)$ for all down-neighbors $z$ of $y$ and all the corresponding edges $e$.
\end{itemize}

\noindent
Obviously, for every vertex $y$ of the tree $T$, we can consider the subtree rooted at $y$ as a connected component of the graph $T - e$, where $e$ is the up-edge of $y$. Therefore, $n_1(e|T)=w(y)$ and $m_1(e|T)=w''(y)$. Let $n(T)=\sum_{u \in V(T)} w(u)$ and $m(T)=\sum_{e \in E(T)} w'(e)$ (this can be computed in linear time). It follows that $n_2(e|T) = n(T) - w(y)$ and $m_2(e|T) = m(T) -w''(y) - w'(e)$. Let $X$ be the sum of numbers $s(z)$ for all down-neighbours $z$ of $y$ (and $X=0$ if $y$ is a leaf). Finally, set $s(y) = X + w'(e)(n_1(e|T) + m_1(e|T))(n_2(e|T) + m_2(e|T))$ if $y \neq x$ and $s(y) = X$, if $y=x$. It is obvious that $s(x) = Sz_t(T,w,w')$ and the proof is complete. \qed
 
\end{proof}

\begin{lemma}
\label{alg2}
Let $(T,w')$ be a weighted tree with $m$ edges. Then the weighted PI index $PI_e(T,w')$ can be computed in $O(m)$ time.
\end{lemma}

\begin{proof}
The proof is very similar to the proof of Lemma \ref{alg1}. \qed
\end{proof}

\begin{lemma}
\label{alg31}
Let $(G,w,w')$ be a weighted bipartite graph and let 
$$n(G)=\sum_{u \in V(G)} w(u), \quad m(G)=\sum_{e \in E(G)} w'(e).$$
Then the weighted vertex-PI index $$PI_v(G,w,w')=n(G)m(G).$$
\end{lemma}

\begin{proof}
Since $G$ is bipartite, for any edge $e \in E(G)$ it holds $V(G) = N_1(e|G) \cup N_2(e|G)$. Therefore, $n_1(e|G) + n_2(e|G) = n(G)$ for any edge $e$. Then it is obvious that
$$PI_v(G,w,w') = \sum_{e \in E(G)}w'(e)n(G) = n(G)m(G).$$
\qed
\end{proof}

\begin{corollary}
\label{alg3}
Let $(T,w,w')$ be a weighted tree. Then the weighted vertex-PI index $PI_v(T,w,w')$ can be computed in $O(m)$ time.
\end{corollary}

\noindent
Finally, using Lemma \ref{alg1}, Lemma \ref{alg2} and Corollary \ref{alg3}, we obtain the final result.
\begin{theorem}
\label{thm:main}
If $G$ is a benzenoid system with $m$ edges, then Algorithm \ref{alg:szeged} and Algorithm \ref{alg:pi} correctly compute $Sz_e(G)$ and $PI_e(G)$, respectively, and can be implemented in $O(m)$ time.
\end{theorem}

\section{An example}
\label{example}

In this final section we give an example that demonstrates how the results from Section \ref{sec:theory} can be used to calculate the indices by hand. Consider the benzenoid system $G$ from Figure \ref{sistem} with $m=25$ edges.

\begin{figure}[H] 
\begin{center}
\includegraphics[scale=0.7]{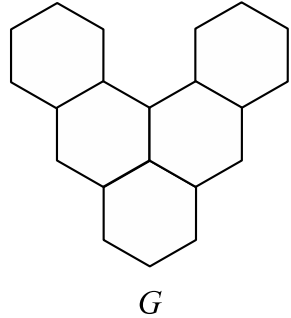}
\end{center}
\caption{\label{sistem} Benzenoid system $G$.}
\end{figure}

\noindent
First, we determine the graphs $G-E_i$ for $i \in \lbrace 1, 2, 3 \rbrace$, where $E_i$ is the set of edges of $G$ of the same direction. These graphs are represented in Figure \ref{components}. Next we determine the weighted quotient trees $(T_i, w_i, w_i')$, $1\le i\le 3$, see Figure \ref{trees}.

\begin{figure}[h!] 
\begin{center}
\includegraphics[scale=0.6]{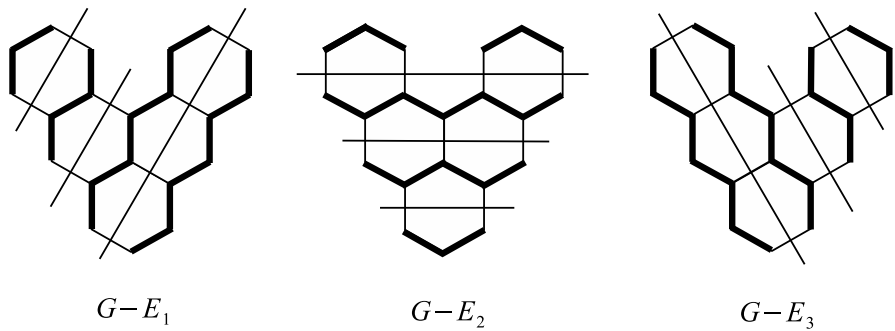}
\end{center}
\caption{\label{components} Graphs $G-E_1$, $G-E_2$, and $G-E_3$.}
\end{figure}

\begin{figure}[h!] 
\begin{center}
\includegraphics[scale=0.6]{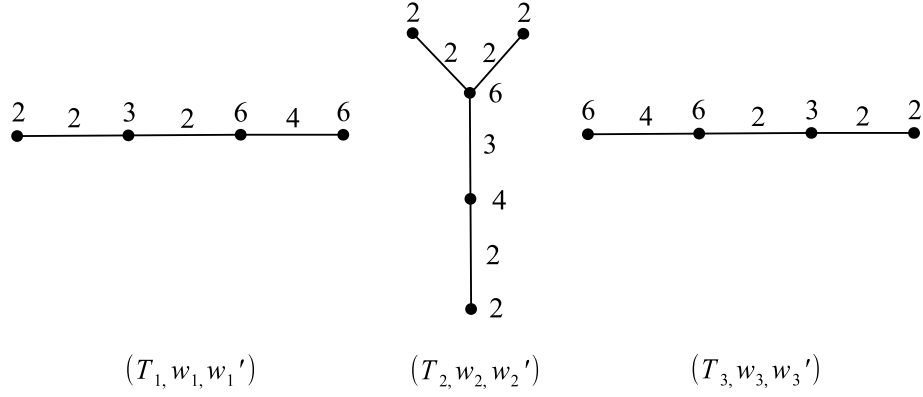}
\end{center}
\caption{\label{trees} Weighted quotient trees.}
\end{figure}

\noindent
Now we are ready to compute the edge-Szeged index. We start with the following quantities:
$$Sz_t(T_1,w_1,w_1') = Sz_t(T_3,w_3,w_3') = 2 \cdot 2 \cdot 21 + 2 \cdot 7 \cdot 16 + 4 \cdot 15 \cdot 6 = 668,$$
$$Sz_t(T_2,w_2,w_2') = 2 \cdot 2 \cdot 21 + 3 \cdot 8 \cdot 14 + 2 \cdot 21 \cdot 2 + 2 \cdot 21 \cdot 2 = 588.$$

\noindent
Therefore, by Corollary \ref{sz_cor},
$$Sz_e(G) = 668 + 588 + 668 = 1924.$$

\noindent
Next, we compute the PI index. We start with the following quantities:
$$PI_e(T_1,w_1') = PI_e(T_3,w_3') = 2 \cdot (0 + 6) + 2 \cdot (2+4) + 4 \cdot (4 + 0) = 40,$$
$$PI_e(T_2,w_2') = 2 \cdot (0 + 7) + 3 \cdot (2 + 4) + 2 \cdot (7 + 0) + 2 \cdot (7 + 0) = 60,$$
$$PI_v(T_1,w_1,w_1') = PI_v(T_3, w_3, w_3') = (2+3+6+6) \cdot (2+2+4) = 136,$$
$$PI_v(T_2,w_2,w_2') = (2+4+6+2+2) \cdot (2+3+2+2) = 144.$$

\noindent
Therefore, by Theorem \ref{pi_sum},
$$PI_e(G) = (40 + 136) + (60 + 144) + (40 + 136) = 556.$$

\section{Generalization to weighted graphs}

The method described in this paper can also be applied to calculate the edge-Szeged index and the PI index of weighted benzenoid systems. More precisely, if $(G,w')$ is an edge-weighted benzenoid system and $T_1$, $T_2$, $T_3$ are its quotient trees, we can extend the quotient trees $T_i$, $i \in \lbrace 1, 2, 3 \rbrace$, to weighted trees $(T_i, u'_i)$ and $(T_i, u_i, u'_i)$ as follows: 
\begin{itemize}
\item for $x \in V(T_i)$, let $u_i(x)$ be the sum of weights of edges in the component (path) $x$ of $G_i$;
\item for $e = xy \in E(T_i)$, let $u_i'(e)$ be the sum of weights of edges between components (paths) $x$ and $y$.
\end{itemize}
With this notation we can prove the following theorem. The proof is omitted since the ideas in the proof are similar as in Section \ref{sec:theory}.
\begin{theorem}
\label{sz_cor}
If $(G,w')$ is an edge-weighted benzenoid system and $(T_i,u_i,u_i')$, $1\le i\le 3$, are the corresponding weighted quotient trees, then
$$ Sz_e(G,w')= Sz_t(T_1,u_1,u_1') + Sz_t(T_2, u_2,u_2') + Sz_t(T_3, u_3, u_3'),$$
$$PI_e(G,w')= \sum_{i=1}^{3}\big( PI_e(T_i,u_i') + PI_v(T_i,u_i,u_i') \big).$$
\end{theorem}
Following this result, fast algorithms for the edge-Szeged index and the PI index of edge-weighted benzenoid systems can be obtained as well.

\end{document}